\documentclass[12pt]{amsart}

\usepackage{mathptmx,amsmath,amsthm,amsfonts,amssymb}
\usepackage[margin=1in,foot=6ex]{geometry}
\usepackage{pgf,tikz}
\usepackage{graphicx}
\usepackage{amsaddr}

\usetikzlibrary{arrows}

\newtheorem{theorem}{Theorem}[section]

\newtheorem{lemma}[theorem]{Lemma}
\newtheorem{proposition}[theorem]{Proposition}
\newtheorem{corollary}[theorem]{Corollary}
\theoremstyle{definition}

\newtheorem{example}[theorem]{Example}

\newcommand{\R}{{\mathbb{R}}}

\def\eps{\varepsilon}
\definecolor{zzttqq}{rgb}{0.6,0.2,0}

\begin{document}

\title{A short proof of Klee's theorem}
\author{John Zanazzi}
\address{
Northern Arizona University\\
Dept. of Mathematics and Statistics\\
Flagstaff, AZ 86011 
}
\address{Penn State University Mathematics Dept.\\ University Park, State College, PA 16802}
\email{jjz54@cornell.edu}
\thanks{\textit{Current address}: 610 Space Science Building, Cornell University, Ithaca, NY 14853}
\thanks{\textit{Phone}: (480)313-7794}

\maketitle

\begin{abstract}

In 1959,  Klee proved that a convex body $K$ is a polyhedron if and only if all of its projections are polygons. In this paper, a new proof of this theorem is given for convex bodies in $\R^3$.

\end{abstract}

\section{Introduction}

This paper will begin by summarizing the relevant work of Mirkil \cite{HMirkil} and Klee \cite{VKlee}.  Let $V$ be a $n$-dimensional real vector space and $C \subset V$.  The set $C$ is said to be a \textit{convex cone} if and only if $C$ is stable under both vector addition and multiplication, and \textit{polyhedral} if and only if $C$ is the intersection of a finite number of closed halfspaces. For a set $K$ embedded in a $n$-dimensional affine space $E$ and a point $p \in K$, define $K$ to be \textit{polyhedral at} $p$ if and only if some neighborhood of $p$ relative to $K$ is polyhedral.  For a set $K \subset E$ and a point $p \in E$, we will denote the smallest cone containing $K$ with vertex $p$ as $\text{cone}(p,K)$.  A $j$-\textit{flat} is a $j$-dimensional affine subspace of $E$, and a \textit{hyperplane} is a $(n-1)$-dimensional affine subspace of $E$.

In Mirkil \cite{HMirkil}, the following theorem is proven:

\begin{theorem}
If $C$ is a closed convex cone, then $C$ is polyhedral if and only if every 2-dimensional projection of $C$ is closed.
\label{thm:Mirkil}
\end{theorem}

\textit{Sketch of Proof}. The forward direction of this statement follows from the fact that every projection of $C$ is polyhedral.  The main idea to prove the converse is as follows:  If $H$ is a hyperplane, then for all $x \in C\cap H$, there exists a neighborhood $N$ which contains no extreme points except possibly $x$.  \hfill\(\Box\)

\begin{example}
Let our vector space be $\R^3$ with the standard Cartesian coordinate system.  Take $C$ to be a circular cone supported by the $(x,y)$ plane so that the infinite half-line of support lies on the $x$-axis, and let $\pi_{(y,z)}(C)$ be the horizontal projection of $C$ into the $(y,z)$ plane. We see that $\pi_{(y,z)}(C)$ may be expressed as
 \[ \pi_{(y,z)}(C) = \{(0,a,b):a\in\R, \; b>0\} \cup \{(0,0,0)\}\]
Note that $\pi_{(y,z)}(C)$ is not closed, in accordance with Theorem \ref{thm:Mirkil}.
\end{example}

 Motivated by Theorem \ref{thm:Mirkil} comes the extensive work of Klee \cite{VKlee}, which includes the following theorem:

\begin{theorem} \label{thm:PolyP}  If $K$ is a $n$-dimensional convex subset of an affine space, $p \in K$, and $2\le j\le n$, then $K$ is polyhedral at $p$ if and only if $\pi(K)$ is polyhedral at $p$ whenever $\pi(K)$ is an affine projection of $K$ into a $j$-flat through $p$. \end{theorem}

\textit{Sketch of Proof}. To prove the ``only if" portion, the fact that a convex set $K$ is polyhedral at point $p\in K$ if and only if $\text{cone}(p,K)$ is polyhedral is used repeatedly on $K$, $\text{cone}(p,K)$, and their affine projections.

To prove the converse, the fact that all $j$-dimensional projections of $K$ are polyhedral follows directly from the statement.  In particular, all 2-dimensional projections of $K$ are polyhedral, thus all 2-dimensional projections of $K$ are closed.  Using Mirkil \cite{HMirkil}, this implies all intersections with hyperplanes are polyhedral.  Furthermore, Klee proves a cone is polyhedral if and only if its intersections with elements of a specific parameterized set of hyperplanes are polyhedral, thus $\text{cone}(p,K)$ is polyhedral.  Using again that a convex set $K$ is polyhedral at point $p\in K$ if and only if $\text{cone}(p,K)$ is polyhedral, Klee proves $K$ is polyhedral at $p$.  \hfill\(\Box\)

From Theorem \ref{thm:PolyP}, Klee establishes a corollary:

\begin{corollary} With $2\le j\le n$, a n-dimensional bounded convex subset is polyhedral if and only if all its projections into $j$-flats are polyhedral.  \end{corollary}

The following statement follows from the previous corollary:

\begin{theorem}\label{thm:Klee}
If $K$ is a convex body in $\R^3$ whose orthogonal projection into every plane is a polygon, then $K$ is a polyhedron.
\end{theorem}

The proof of this theorem is simplified if instead of reformulating the problem in terms of closed projections of convex cones, one shows that all points $x\in K$ are located within a neighborhood which contain no extreme points except possibly $x$.  The next sections explain this proof in detail.

\section{Dual reformulation}

For a plane $P$ and sets $X,Y$ embedded in $\R^3$, denote the orthogonal projection of $X$ into $P$ by $\pi_P(X)$, the union and intersection of $X$ and $Y$ by $X\cup Y$ and $X \cap Y$ respectively, the convex hull of $X$ by $\text{conv}(X)$, and the boundary of $X$ by $\partial X$.  For points $p,q,r$ in $\R^3$, denote the triangle with vertices $p,q,r$ by $\triangle pqr$, and the line segment bounded by $p$ and $q$ by $[pq]$.

Recall that a \emph{convex body} is a closed bounded convex set with nonempty interior.  Fix a convex body $K$ in $\R^3$ so that the origin of $\R^3$ belongs to the interior of $K$.  The \emph{polar dual} of $K$ will be denoted as $K^*$;
i.e.,
\[ K^* = \{\, y \mid x \cdot y \le 1\ \text{for every}\  x \in K \,\} \]

Clearly $K^*$ is a convex body
and 
the origin is an interior point of $K^*$.
Moreover $K^*$ is a convex polyhedron if and only if so is $K$.

The following statement follows directly from the definition of polar dual.

\begin{proposition} If $P$ is a plane passing through the origin, then 
\[K^* \cap P= \pi_P(K)^*\cap P.\]
\label{th: PDual} \end{proposition}

Note that $\pi_P(K)^*\cap P$ is a polygon if and only if so is $\pi_P(K)$.
Using the above proposition, Theorem \ref{thm:Klee} can be reformulated the following way:

\begin{theorem}\label{thm:reformulation}
Suppose $K^*$ is a convex body in $\R^3$
containing the origin in its interior.  If for every plane $P$ passing through the origin, the intersection $P\cap K^*$ is a polygon, then $K^*$ is a polyhedron.
\end{theorem}

\section{Proof of Theorem \ref{thm:reformulation}}

\begin{lemma}\label{lem:3-edges}
Let $K^*$ be a convex body in $\R^3$ and $p,q,x,y\in K^*$.
If $x$ lies between $p$ and $q$, and the line
segment $[xy]$
lies completely in $\partial K^*$,
then the triangle $\triangle pqy$ lies completely in $\partial K^*$.
\end{lemma}

\begin{proof}
Suppose to the contrary that the point $r \in \triangle pqy$ belongs to the interior of $K^*$.  This implies the existence of a line segment $L \subset K^*$ containing $r$ such that the convex hull $\text{conv}(L \cup \triangle pqy)$ is a bipyramid in $\R^3$, and the interior of $\triangle pqy$ lies in the interior of $\text{conv}(L \cup \triangle pqy)$.

Therefore, all the interior points of $\triangle pqy$ belong to the interior of $K^*$.  Because the midpoint of $[xy]$ lies in the interior of $\triangle pqy$, the result follows.
\end{proof}

\begin{proposition}\label{prop:main}
Suppose $K^*$ satisfies the conditions of Theorem \ref{thm:reformulation}.  For all points $p,q \in K^*$, there exists an $\eps>0$ such that for $r \in K^*$,
if $ \, 0<|p-r|<\eps$ 
and $\angle rpq<\eps$,
then
$r$ is not extreme.
\end{proposition}

\begin{proof}
The statement is evident if the line segment $[pq]$ passes through the interior of $K^*$,
so we can assume that $[pq]\subset\partial K^*$.

Let $x$ denote the midpoint of $[pq]$.
Choose a plane $P$ through the origin which intersects $[pq]$ transversely at $x$.  The intersection $K^*\cap P$ is a polygon, where the sides extending from $x$ are denoted by the line segments $[xy]$ and $[xz]$.  We refer to Figure \ref{f:intersect} for clarity.

\begin{figure}
\begin{tikzpicture}[line cap=round,line join=round,>=triangle 45,x=1.0cm,y=1.0cm]
\clip(-3.7,1.94) rectangle (1.84,5.86);
\fill[color=zzttqq,fill=zzttqq,fill opacity=0.1] (-3.22,5.08) -- (-2.22,2.6) -- (1.18,3.66) -- cycle;
\fill[color=zzttqq,fill=zzttqq,fill opacity=0.1] (-0.28,5.14) -- (-3.22,5.08) -- (1.18,3.66) -- cycle;
\draw [line width=2.8pt] (-3.22,5.08)-- (1.18,3.66);
\draw [line width=2pt] (-1.27,4.45)-- (-2.22,2.6);
\draw [line width=2pt] (-1.27,4.45)-- (-0.28,5.14);
\draw [color=zzttqq] (-3.22,5.08)-- (-2.22,2.6);
\draw [color=zzttqq] (-2.22,2.6)-- (1.18,3.66);
\draw [color=zzttqq] (1.18,3.66)-- (-3.22,5.08);
\draw [color=zzttqq] (-0.28,5.14)-- (-3.22,5.08);
\draw [color=zzttqq] (-3.22,5.08)-- (1.18,3.66);
\draw [color=zzttqq] (1.18,3.66)-- (-0.28,5.14);
\begin{scriptsize}
\fill [color=black] (-3.22,5.08) circle (1.5pt);
\draw[color=black] (-3.44,5.38) node {$p$};
\fill [color=black] (1.18,3.66) circle (1.5pt);
\draw[color=black] (1.52,3.72) node {$q$};
\fill [color=black] (-1.27,4.45) circle (1.5pt);
\draw[color=black] (-1.68,4.32) node {$x$};
\fill [color=black] (-2.22,2.6) circle (1.5pt);
\draw[color=black] (-2.4,2.38) node {$y$};
\fill [color=black] (-0.28,5.14) circle (1.5pt);
\draw[color=black] (-0.04,5.42) node {$z$};
\end{scriptsize}
\end{tikzpicture}
\caption{}
\label{f:intersect}
\end{figure}

By Lemma~\ref{lem:3-edges}, the triangles $\triangle pyq$ and $\triangle pzq$ lie completely in $\partial K^*$.  Choose a point $s$ in the interior of $K^*$.  Clearly there exists an $\epsilon >0$ such that for any point $r\in K^*$, if $|p-r|<\epsilon$ and $\angle rpq<\epsilon$, then $r$ lies in the convex hull $\text{conv}(\triangle pyq,\triangle pzq,s)$.  Hence the result follows.
\end{proof}

\begin{proof}[Proof of Theorem \ref{thm:reformulation}]

Suppose to the contrary that $\{q_n\}$ is an infinite set of distinct extreme points contained within $K^*$.  Pass $\{q_n\}$ to a convergent subsequence $\{q_{n_k}\}$, and let $p \in \partial K^*$ be the point such that $q_{n_k} \to p$ as $n_k \to \infty$.  Choose the convergent subsequence $\{q_{n_k}\}$ so that the unit vectors $v_{n_k}=\tfrac{q_{n_k}-p}{|q_{n_k}-p|}$
also converge,
say $v_{n_k} \to u$. 

Consider the plane $P$ which passes through  $p$, $u$ and the origin.
Since the intersection $P\cap K^*$ is a polygon,
there is a line segment $[pq]\subset\partial K^*$
pointing from $p$ in the direction of $u$.

Applying Proposition~\ref{prop:main}, 
we arrive at a contradiction.
\end{proof}

\section{Acknowledgments}

I am greatly indebted to Anton Petrunin, Greg Kuperberg, and Sergei Tabachnikov for suggestions and editing drafts.  This paper would not have been possible without them.  I would like to thank Branko Grunbaum for suggestions, as well as Pennsylvania State University and the National Science Foundation for funding this research through a Mathematics Advanced Study Semesters (MASS) Fellowship.

\end{document}